\documentclass[a4paper, 12pt,reqno]{amsart}

\usepackage{amsmath}
\usepackage{amssymb}
\usepackage{hyperref} 
\usepackage{amsfonts}
\usepackage{mathtools}
\usepackage{theoremref}
\usepackage{amsthm}
\usepackage[dvipsnames]{xcolor}
\usepackage[utf8]{inputenc}
\usepackage{array}
\usepackage{enumitem}
\usepackage[english]{babel}
\usepackage{comment}
\usepackage{multicol}

\usepackage{tikz}
\usetikzlibrary{shapes.geometric}
\usetikzlibrary {arrows.meta}

\usetikzlibrary{hobby}
\usetikzlibrary{positioning}

\calclayout

\setlist{topsep=0mm,partopsep=0mm,itemsep=1mm}

\theoremstyle{plain}
\newtheorem{thm}{Theorem}[section]
\newtheorem{cor}[thm]{Corollary}

\newtheorem{prop}[thm]{Proposition}

\theoremstyle{definition}

\newtheorem{example}[thm]{Example}

\newtheorem{definition}[thm]{Definition}

\theoremstyle{remark}

\newtheorem{cclaim}{Claim}[thm]

\newenvironment{claim}{\begin{cclaim}\it}{\end{cclaim}}

\newcommand{\N}{\mathbb{N}}
\newcommand{\Z}{\mathbb{Z}}

\newcommand{\R}{\mathcal{R}}
\renewcommand{\L}{\mathcal{L}}

\newcommand{\J}{\mathcal{J}}

\newcommand{\BR}{\textup{\textsf{BR}}}
\newcommand{\B}{\mathcal{B}}

\newenvironment{thmenumerate}{\begin{enumerate}[label=\textup{(\roman*)},leftmargin=10mm]}{\end{enumerate}}
\newenvironment{nitemize}{\begin{itemize}[label=\textbullet, leftmargin=5mm]}{\end{itemize}}

\newcommand{\CS}{\mathcal{C}^1(S;\alpha,\beta;P)}

\title{Semigroup Congruences and Subsemigroups of the Direct Square}

\author[C.\ Barber]{Callum Barber}
\address{School of Mathematics and Statistics, University of St Andrews, St Andrews, Fife KY16 9SS, UK.}
\email{cjb38@st-andrews.ac.uk}

\author[N.\ Ru\v{s}kuc]{Nik Ru\v{s}kuc}
\address{School of Mathematics and Statistics, University of St Andrews, St Andrews, Fife KY16 9SS, UK.}
\email{nik.ruskuc@st-andrews.ac.uk}



\begin{document}

\maketitle

\begin{abstract}
We investigate semigroups $S$ which have the property that every subsemigroup of $S\times S$ which contains the diagonal $\{ (s,s)\colon s\in S\}$ is necessarily a congruence on $S$. 
We call such $S$  a  DSC semigroup.
It is well known that all finite groups are DSC, and easy to see that every DSC semigroup must be simple. Building on this, we show that for broad classes of semigroups -- including periodic, stable, inverse and several well-known types of simple semigroups -- the only DSC members are groups. However, it turns out that there exist non-group DSC semigroups, which we obtain utilising a construction introduced by Byleen for the purpose of constructing interesting congruence-free semigroups. Such examples can additionally be regular or bisimple.
\end{abstract}

\section{Introduction}

Given an algebra $A$, a congruence on $A$ is an equivalence relation that is compatible with the operations of the algebra. 
We can also think of $\rho$ as a subset of the direct product $A \times A$.
So instead of $\rho$ being reflexive we can think of $\rho$ as containing the diagonal $\Delta = \{(x,x) \colon  x \in A \}$, and the notion of $\rho$ respecting the operations then becomes $\rho$ being a subalgebra of $A \times A$. 
Motivated by this we give the following definition.

\begin{definition}
    Let $A$ be an algebra. A \textit{diagonal subalgebra} $\rho$ of $A \times A$ is a subalgebra of $A \times A$ that contains the diagonal $\Delta = \{(x,x) \colon  x \in A \}$. A \textit{congruence} on $A$ is a diagonal subalgebra of $A \times A$, such that for all $x,y,z \in A$ we have:
\[
        (x,y) \in \rho \Rightarrow (y,x) \in \rho \quad\textup{and}\quad (x,y),(y,z) \in \rho \Rightarrow (x,z) \in \rho.
\]
\end{definition}

It is a well known, easy fact that for groups diagonal subgroups and congruences are one and the same.

\begin{prop}
    \label{Motivation Theorem}
    Let $G$ be a group. Then the diagonal subgroups of $G \times G$ are precisely the congruences on $G$.
\end{prop}
\begin{proof}
This is regarded as folklore, but we provide a short proof for completeness.
    By definition any congruence is a diagonal subgroup. If $\rho$ is a diagonal subgroup, and if $(x,y),(y,z) \in \rho$, then, bearing in mind that $(x,x),(y,y),(y^{-1},y^{-1})\in \rho$, we have:
    \[        (y,x)  = (y,y)(x,y)^{-1}(x,x) \in \rho \quad\text{and}\quad
        (x,z)  = (x,y)(y^{-1},y^{-1})(y,z) \in \rho.
   \]
    Hence $\rho$ is a congruence. 
\end{proof}

The same result holds more general for any algebras $A$ with a Mal'cev term, i.e. term $m(x,y,z)$ in three variables such that
$m(x,y,y)=x=m(y,y,x)$ holds for all $x,y\in A$; see for example \cite[Theorem 4.70]{MM18}.
In particular, the result holds for rings, associative and Lie algebras, loops and quasigroups.
However it does not hold for semigroups, as the following easy example shows:

\begin{example}
    Consider the left zero semigroup $S = \{ x,y \}$ with multiplication $ab = a$ for all $a,b \in S$. The set $\rho = \{ (x,x) , (x,y) , (y,y) \}$ is a diagonal subsemigroup of $S \times S$ but is not a congruence on $S$.
\end{example}

Motivated by this we give the following definition.
\begin{definition}
    \label{DSC}
    We will say that a semigroup is \textit{DSC} if every diagonal subsemigroup is a congruence.
\end{definition}

Over the course of this paper we will see that DSC semigroups are few and far between. 
In fact, with any of a number of additional mild assumptions, the only DSC semigroups are groups.
A further wrinkle worth keeping in mind is that, despite Proposition \ref{Motivation Theorem}, not even all groups are DSC, due to the fact that a group may contain subsemigroups that are not subgroups. Here is a concrete example

\begin{example}
Let $\Z$ denote the infinite cyclic group. Then $\{ (x,y) \in \Z \times \Z \colon  x \leq y \}$ 
is a diagonal subsemigroup of $\Z\times\Z$, but is not a congruence.
\end{example}

Of course, this `anomaly' cannot arise for finite, or indeed periodic, groups.

We can prove, in full generality, that all DSC semigroups are simple (Theorem \ref{simple semigroup}).
Proceeding from there, we prove for a semigroup $S$:

\begin{nitemize}
\item
Supposing $S$ is finite or periodic, $S$ is DSC if and only if $S$ is a group; Corollaries~\ref{finite semigroup},~\ref{periodic semigroup}.
\item
If $S$ is a stable or inverse DSC semigroup then $S$ is a group; Corollary~\ref{stable semigroup}, Theorem~\ref{Inverse semigroup}.
\end{nitemize}

Focusing on special classes of simple semigroups, we also have:

\begin{nitemize}
\item
If $S$ is a completely simple DSC semigroup then $S$ is a group; Theorem \ref{Completely simple}.
\item
For any semigroup $S$ and any endomorphism $\theta\colon S\rightarrow S$, the Bruck--Reilly extension $\BR(S,\theta)$ is not DSC; Theorem \ref{bruck-reilly}.
\item
For any two infinite cardinals $p \geq q$, the generalized Baer--Levi semigroup $\B(p,q)$ is not DSC;
Theorem \ref{Baer Levi}.
\end{nitemize}

In Theorem \ref{quotient} we will prove that the class of DSC semigroups is closed under quotients.
Thus, one might wish to look for possible non-group examples among simple, congruence-free semigroups.
Byleen \cite{Byleen} gives a construction which, under certain conditions, yields such semigroups.
It turns out that we can deploy this construction to show that
    \begin{nitemize}
        \item There exist non-group DSC semigroups; Corollary \ref{cor:D} \ref{it:D1}.
        \item Furthermore, there are such examples that are regular and bisimple; Corollary~\ref{cor:D}~\ref{it:D2}.
    \end{nitemize}

As a byproduct we also observe that the class of DSC semigroups is not closed under subsemigroups;
Corollary \ref{cor:D} \ref{it:D3}.

We will require only very basic concepts from semigroup theory. They will be introduced within the text where they are needed first. For a more systematic introduction we refer the reader to any standard monograph such as \cite{Ho95}.
We will use $\N$ to denote the set of all positive integers, and $\N_0$ for $\N\cup\{0\}$.

\section{Completely Simple, Stable and Inverse semigroups.}
\label{sec:SSI}

In this section we will show that all DSC semigroups belonging to certain classes  are in fact groups.
Specifically, we will do this for completely simple, stable and inverse semigroups, in that order.

A non-empty subset $I$ of a semigroup $S$ is said to be an \emph{ideal} if for all $x\in I$ and all $s\in S$ we have $sx,xs\in I$.
A semigroup is said to be \emph{simple} if it has no ideals other than itself.

\begin{thm}\label{simple semigroup}
    Any DSC semigroups is simple.
\end{thm}

\begin{proof}
    Suppose $S$ is not simple. Let $I$ be a proper ideal of $S$. It is easily seen that $\rho = I \times S \cup  \Delta$ is a diagonal subsemigroup of $S \times S$. If we take $x \in I$ and $y \in S \setminus I$ then $(x,y) \in \rho$ but $(y,x) \notin \rho$. Hence $\rho$ is not a congruence and $S$ is not a DSC semigroup.
\end{proof}

Let $S$ be a semigroup, and let $E$ be the set of idempotents of $S$.
The relation $\leq$ on $E$ defined by $e\leq f \iff ef=fe=e$ is a partial order.
Any minimal element in this partial order is said to be \emph{primitive}.
A simple semigroup $S$ is said to be \emph{completely simple} if it has a primitive idempotent. 
All finite simple semigroups are completely simple.

There is a complete structural description of completely simple semigroups, originally due to Suschkewitsch
\cite{Su28}.
Let $G$ be a group, let $I$ and $J$ be two index sets, and let $P=(p_{ji})_{j\in J,i\in I}$ be a $J\times I$ matrix with entries from $G$.
The \emph{Rees matrix semigroup} $\mathcal{M}[G;I,J;P]$ is the set $I\times G\times J$
with multiplication 
$(i , g , j)(k , h , l) = (i, g p_{jk} h , l)$.
Suschkewitsch's Theorem then asserts that a semigroup $S$ is completely simple if and only if it is isomorphic to some Rees matrix semigroup $\mathcal{M}[G;I,J;P]$; see \cite[Theorem~3.3.1]{Ho95}.

\begin{thm}\label{Completely simple}
    Let $S$ be a completely simple semigroup. If $S$ is DSC then $S$ is a group.
\end{thm}

\begin{proof}
    Let $S$ be a completely simple semigroup.
    By, Suschkewitsch's Theorem without loss of generality we may assume $S = \mathcal{M}[G;I,J;P]$.
    If $|I| > 1$ then pick $i \neq k \in I$. Now consider the set:
        \begin{align*}
        \rho & = \bigl\{  \bigl((i,g,j),(k,h,j)\bigr) \colon g,h\in G,\ j\in J\bigr\}\\ 
        &\cup \bigl\{ \bigl(l,g,j),(l,h,j)\bigr) \colon l \in I,\ g,h \in G ,\ j \in J \bigr\}.
    \end{align*}
    It is routine to verify that $\rho$ is a diagonal subsemigroup.
    For an arbitrary $j\in J$  it is easily seen that $\bigl((i,1,j),(k,1,j)\bigr) \in \rho$ but 
    $\bigl((k,1,j),(i,1,j)\bigr) \notin \rho$. 
    Hence $\rho$ is not a congruence, contradicting $\mathcal{M}[G;I,J;P]$ being DSC. 
    Therefore it must then be the case that $|I| = 1$ and, analogously, $|J| = 1$.
    It now easily follows that $S\cong G$, a group.
\end{proof}

We know from Theorem \ref{simple semigroup} that a DSC semigroup $S$ must be simple. 
Whenever we can show that under some additional assumptions $S$ must in fact be completely simple, Theorem \ref{Completely simple} will force $S$ to be a group. 
We deploy this strategy for stable and inverse semigroups.

In order to define stability, we need to introduce Green's equivalences $\R$, $\L$ and $\J$ on a semigroup $S$:
\[
s\R t \Leftrightarrow sS^1=tS^1,\quad s\L t \Leftrightarrow S^1s=S^1t,\quad
s\J t \Leftrightarrow S^1sS^1=S^1tS^1;
\]
for a detailed introduction see \cite[Section 2.1]{Ho95}.
We then say that $S$ is \emph{stable} if the following implications hold:
\[
x \J sx \Rightarrow x \L sx \quad\text{and}\quad x \J xs \Rightarrow x \R xs \qquad\text{for all 
$s,x\in S$.}
\]
All finite semigroups are stable \cite[Theorem A.2.4]{RS09}.  

By \cite[Theorem A.4.15]{RS09} every stable simple semigroup is completely simple, and so we have:

\begin{cor}\label{stable semigroup}
    Let $S$ be a stable semigroup. If $S$ is DSC then $S$ is a group. \qed
\end{cor}

A semigroup $S$ is said to be \emph{periodic} if for every $s\in S$ there exist distinct $m,n\in\N$ such that $s^m=s^n$. Every finite semigroup is periodic. 

\begin{cor}\label{periodic semigroup}
    Let $S$ be a periodic semigroup. Then $S$ is DSC  if and only if $S$ is a group.
\end{cor}

\begin{proof}
($\Rightarrow$)
By \cite[Theorem A.2.4]{RS09}, every finite semigroup is stable. In fact, the proof is valid under the weaker assumption of periodicity; see also proof of \cite[Corollary 3.1]{An65}. 
This direction now follows from Corollary \ref{stable semigroup}.

($\Leftarrow$)
    Suppose $S$ is a group. Let $\rho$ be a diagonal subsemigroup of $S \times S$. We claim that $\rho$ is also a diagonal subgroup of $S \times S$, and the result then follows from Proposition \ref{Motivation Theorem}. Let $(x,y) \in \rho$. As $S$ is periodic, $x^a = 1 = y^b$ for some $a,b \in \N$. Then $(x,y)^{-1} = (x^{-1},y^{-1}) = (x^{ab - 1},y^{ab - 1}) = (x,y)^{ab-1} \in \rho$.
\end{proof}

\begin{cor}\label{finite semigroup}
    Let $S$ be a finite semigroup. Then $S$ is DSC  if and only if $S$ is a group.
    \qed
\end{cor}

We finish this section with a discussion of inverse semigroups.
An element $s\in S$ is said to be \emph{regular} if $sts=s$ for some $t\in S$.
If in addition $tst=t$ we say that $s$ and $t$ are (semigroup) inverses of each other.
It is known that an element is regular if and only if it has an inverse; see \cite[p.~51]{Ho95}.
The semigroup $S$ is \emph{regular} if every element is regular, and it is \emph{inverse} if every element has a unique inverse.

\begin{thm}\label{Inverse semigroup}
    Let $S$ be an inverse semigroup. If $S$ is DSC then $S$ is a group.
\end{thm}

\begin{proof}
    Let $S$ be an inverse DSC semigroup. There is a natural partial order on $S$ given by $x \leq y \iff x = ey$ for some idempotent $e$; see \cite[Section 5.2]{Ho95}. It is also true that this partial order is compatible with the multiplication and restricts to the natural partial order on the idempotents.
    So $\rho = \{ (x,y) \in S \times S \colon  x \leq y \}$ is a diagonal subsemigroup and by assumption is a congruence. Hence $\leq$ is both symmetric and antisymmetric, and therefore there exists a primitive idempotent.
\end{proof}

We started this paper by introducing the concept of diagonal subalgebra for general algebras and we have looked at the cases when the algebra is a group or semigroup.
At this point we have enough theory to ask this question for inverse semigroups as well.

\begin{thm}\label{Diagonal Inverse Subsemigroup}
    Let $S$ be an inverse semigroup. Then every diagonal inverse subsemigroup of $S \times S$ is a congruence if and only if $S$ is a group.
\end{thm}

\begin{proof}
    The proof of $(2) \Rightarrow (1)$ is identical to the proof of Proposition \ref{Motivation Theorem}. The proof of $(1) \Rightarrow (2)$ is similar to the proof of Theorem \ref{Inverse semigroup}. The only difference is that we need to show $\rho$ is a diagonal inverse subsemigroup, which follows from the fact that $x \leq y \Rightarrow x^{-1} \leq y^{-1}$.
\end{proof}

\section{Some Further Infinite Non-DSC Semigroups}

As we have seen in the previous section there exist no non-group completely simple, stable or inverse DSC semigroups.
So if we want to find a non-DSC semigroup we will have to look a bit harder. 
We know that any DSC semigroup is simple, and we will explore different constructions leading to examples of simple semigroups.

One such is the Rees matrix semigroup construction, which we have already encountered, but which can be deployed in greater generality. Specifically, instead of starting with a group $G$, we can start with an arbitrary semigroup $S$. Keeping the remainder of the definition from Section \ref{sec:SSI} unchanged, we obtain the Rees matrix semigroup
$\mathcal{M}[S;I,J;P]$.
It is easy to check, e.g.\ by using \cite[Corollary 3.1.2]{Ho95}, that $\mathcal{M}[S;I,J;P]$ is simple if and only if $S$ is simple. 
By an analogous proof of Theorem \ref{Completely simple}, we can see that if $S^\prime = \mathcal{M}[S;I,J;P]$ is DSC then $|I| = |J| = 1$.
Hence $S^\prime $ has multiplication $x \cdot y = xay$ for some $a \in S$.
We claim that if $S^\prime$ is DSC then $S$ must also be DSC.
Let $\rho$ be a diagonal subsemigroup of $S$, define $\rho^\prime \subseteq S^\prime \times S^\prime$ by $(x,y) \in \rho^\prime \iff (x,y) \in \rho$. 
As $\rho$ contains the diagonal, so does $\rho^\prime$. 
And if we have $(x,y),(z,t) \in \rho^\prime$ then $(x,y),(z,t) \in \rho$, which gives $(xaz,yat) \in \rho$ and hence $(x\cdot z,y \cdot t) \in \rho^\prime$. 
Hence $\rho^\prime$ is a diagonal subsemigroup and, as $S^\prime$ is DSC, a congruence on $S^\prime$.
It follows that $\rho$ is also a congruence on $S$, and therefore $S$ must be DSC.

We have just seen that in order for $\mathcal{M}[S;I,J;P]$ to be DSC, $S$ must also be DSC. 
So for finding more DSC semigroups, taking a Rees matrix semigroup will not be much help as we would have to know whether our original semigroup was DSC to begin with. 

Another way to construct simple semigroups is the Bruck--Reilly extension \cite[Section~5.6]{Ho95}. Here we take a semigroup $S$ and an endomorphism $\theta$  of $S$. The \emph{Bruck--Reilly extension} $\BR(S,\theta)$ is the set $\N_0 \times S \times \N_0$ with multiplication
\begin{equation*}
    (m,s,n)(p,t,q) = (m - n + k , (s \theta^{k-n})(t \theta^{k - p}) , q - p + k)\quad
    \text{where }k = \max(n,p).
\end{equation*}
Under certain conditions the semigroup $\BR(S,\theta)$ is simple.
The conditions themselves will not concern us, but the reader can consult Proposition 5.6.6 and Exercise 5.25 in~\cite{Ho95} for two examples.
Unfortunately, again, this construction will not works for us.
One can show directly that \emph{no} Bruck---Reilly extension is DSC, but it is easier to use some of the theory we have built up in the previous section.
Additionally, we will need the following result.

\begin{thm}\label{quotient}
    Let $S$ be a DSC semigroup and let $\sigma$ be a congruence on $S$. Then the quotient $S/\sigma$ is also DSC.
\end{thm}

\begin{proof}
    Let $\rho$ be a diagonal subsemigroup of $S/\sigma \times S / \sigma$.
    Define $\rho^\prime \subseteq S \times S$ by $(x,y) \in \rho^\prime \iff (x\sigma , y\sigma) \in \rho$. 
    For each $x \in S$, $(x\sigma,x \sigma) \in \rho$ so $(x,x) \in \rho^\prime$. 
    If $(x,y),(z,t) \in \rho^\prime$ then $(x \sigma ,y \sigma), (z \sigma, t\sigma) \in \rho$. 
    This implies $(xz \sigma,yt \sigma) \in \rho$ and so $(xz,yt) \in \rho^\prime$. Hence $\rho^\prime$ is a diagonal subsemigroup of $S \times S$ which by assumption is also a congruence. 
    Now
    \begin{align*}
        &(x \sigma,y \sigma) \in \rho \Rightarrow (x,y) \in \rho^\prime \Rightarrow (y,x) \in \rho^\prime \Rightarrow (y \sigma , x \sigma ) \in \rho, \\
       & (x \sigma , y \sigma) , (y \sigma , z \sigma) \in \rho \Rightarrow (x,y),(y,z) \in \rho^\prime \Rightarrow (x,z)\in \rho^\prime \Rightarrow (x \sigma,z \sigma) \in \rho.
    \end{align*}
    Hence $\rho$ is a congruence and $S/\sigma$ is DSC.
\end{proof}

In particular, the homomorphic image of a DSC semigroup is DSC. This results makes it much easier to check if a semigroup is not DSC.

\begin{thm}\label{bruck-reilly}
    For any semigroup $S$ and endomorphism $\theta$ of $S$, the Bruck--Reilly extension $\BR(S,\theta)$ is not DSC.
\end{thm}

\begin{proof}
Let $B$ denote the \emph{bicyclic monoid}, the semigroup with underlying set $\N_0 \times \N_0$ and multiplication $(m,n)(p,q) = (m-n+k,q-p+k)$, where $k = \max(n,p)$; see \cite[Section 1.6]{Ho95}. 
    The bicyclic monoid is a homomorphic image of $\BR(S,\theta)$ via the projection onto the first and third coordinates.
    It is known that $B$ is an inverse semigroup but not a group; e.g. see \cite[Section 5.4]{Ho95}.
    So, by Theorem \ref{Inverse semigroup}, $B$ is not DSC.
    Hence $\BR(S,\theta)$ is not DSC either, by Theorem~\ref{quotient}.
\end{proof}

The final kind of simple semigroups we will look at in this section are the generalized \emph{Baer--Levi semigroups}; see \cite[Section 8.1]{CP67v2}. 
Here we take two infinite cardinals $p$ and $q$ with $p \geq q$. 
Let $X$ be a set of cardinality $p$ and consider the set 
\[
B(p,q) = \{ f\colon X \to X\: \colon\: f \text{ is injective and } | X \setminus Xf| = q \}.
\]
Under the usual composition of functions $B(p,q)$ forms a semigroup. 
This semigroup is always  simple (even right simple) and contains no idempotents \cite[Theorem 8.2]{CP67v2}, so potentially makes a good candidate for a non-group DSC semigroup. 
Unfortunately, this hope too turns out to be unjustified:

\begin{thm}\label{Baer Levi}
    For any two infinite cardinals $p \geq q$, the generalized Baer--Levi semigroup $B(p,q)$ is not DSC.
\end{thm}

\begin{proof}
    Let $B = B(p,q)$.
    Consider the set 
    \[
    \rho = \bigl\{ (f,g) \in B \times B \colon (X \setminus X f) \cap (X \setminus X g) \neq \emptyset \bigr\}.
    \]
     It is easy to check that $\rho$ is a diagonal subsemigroup of $B \times B$. It is also clear that $\rho$ is symmetric, so we will show $\rho$ is not transitive.
    Partition $X$ into two sets $A$ and $B$ with cardinality $p$. Let $A^\prime$ and $ B^\prime$ be subsets of $A$ and $B$ respectively both with cardinality $q$.
    Let $x \in A^\prime$. The sets $X \setminus A^\prime , X \setminus (B^\prime \cup {x} )$ and $X \setminus B^\prime$ all have cardinality $p$. 
    So there are bijections 
    \[
    f^\prime\colon X \to X \setminus A^\prime ,\quad  g^\prime \colon X \to X \setminus (B^\prime \cup {x})\quad \text{and}\quad h^\prime \colon  X \to X \setminus B^\prime.
    \] 
    Each can be extended to an injection from $X$ to itself; call these functions $f,g$ and $h$ respectively.
    Note that $X\setminus Xf = A^\prime, X\setminus Xg = B^\prime \cup \{ x \}$ and $X\setminus Xh = B^\prime$. So $f,g,h \in B$ and $(f,g),(g,h) \in \rho$, but $(f,h) \notin \rho$. Hence $\rho$ is not transitive.
\end{proof}

\section{ Non-Group DSC Semigroups}

In Theorem \ref{quotient} we saw that any quotient of a DSC semigroup must be DSC.
So if we have a semigroup $S$ with congruence $\sigma$, for $S$ to be DSC so must $S/\sigma$. 
This gives us an extra constraint on being DSC.
So we will try looking at congruence free semigroups.
One rather general such construction was introduced by Byleen in \cite{Byleen}.
The construction uses the notions of monoid actions and presentations, which we now briefly review.

Let $S$ be a monoid with identity $1$, and let $A$ be a set. A \emph{right action} of $S$ on $A$ is a function
$A\times S\rightarrow A$, $(a,s)\mapsto a\triangleright s$ such that
$(a\triangleright s)\triangleright t=a\triangleright (st)$ and $a\triangleright 1=a$ for all $a\in A$ and all $s,t\in S$.
The action is said to be \emph{faithful} if for any two distinct $s,t\in S$ there exists $a\in A$ such that
$a\triangleright s\neq a\triangleright t$.
A \emph{left action} of $S$ on a set $B$ is defined analogously.
For more details see  \cite[Section 8.1]{Ho95}.

Now suppose that $X$ is an alphabet, and denote by $X^\ast$ the \emph{free monoid} on $X$; it consists of all words over $X$, including the empty word $\epsilon$, and the operation is concatenation. 
A \emph{monoid presentation} is a pair of the form $\langle X\mid R\rangle$, where $R\subseteq X^\ast\times X^\ast$.
The \emph{monoid defined} by this presentation is $S=X^\ast/\rho$, where $\rho$ is the congruence generated by $R$. The elements of this semigroup are the $\rho$-classes $[u]$, $u\in X^\ast$.
An \emph{elementary sequence} with respect to $\langle X\mid R\rangle$ is any sequence 
$w_1,w_2,\dots, w_n$ ($n\geq 1$) of words from $X^\ast$ such that for every $i=1,\dots,n-1$ we
have $w_i=w'uw''$, $w_{i+1}=w'vw''$ for some $w',w''\in X^\ast$ and some $(u,v)\in R$ or $(v,u)\in R$.
For two words $u,v\in X^\ast$ we have $[u]=[v]$ if and only if there exists an elementary sequence starting at $u$ and ending in $v$.
We will abuse notation and write $u$ instead of $[u]$ for a typical element of $S$, and $u=v$ instead of $(u,v)$ for a typical element of $R$.
For a more detailed basic introduction to presentations see \cite[Section 1.6]{Ho95}.

\begin{definition}
    Let $S$ be a monoid with identity $1$, and let $A$ and $B$ be sets that are disjoint from each other and from $S$.
    Let $\alpha \colon  A\times S\rightarrow A, (a,s)\mapsto a\triangleright s$ be a right action, and let $\beta\colon  S \times B \rightarrow B, (s,b) \mapsto s \triangleleft b$ be a left action.
    Let $W = A \cup B \cup S$ and let $P$ be a $A \times B$ matrix with entries in $W$.
    Let $\CS$ denote the monoid with monoid presentation:
    \begin{equation*}
        \bigl\langle W \mid ab = p_{a,b} ,\ as = a \triangleright s ,\ sb = s \triangleleft b ,\ st = s \cdot t ,\ 1 = \epsilon \quad (a \in A ,\ b \in B ,\ s,t \in S) \bigr\rangle.
    \end{equation*}
\end{definition}

In the above presentation, the relation $st=s \cdot t$ should be interpreted as a word of length~$2$, namely $st$, being equal to a word of length $1$, the product of $s$ and $t$ in $S$.
In other words, those relations represent the inclusion of the Cayley table of $S$ in the defining presentation for $\CS$.

In \cite{Byleen} it is shown that any element of $\CS$ can be written uniquely in the form $vsu$ where $v \in B^\ast, s \in S$ and $ u \in A^\ast$. 
The monoid $\CS$ has identity $1 = \epsilon$.  
Calculations are easy in $\CS$ as each relation (other than $1 = \epsilon$) replaces a word of length $2$ with a word of length $1$.
In general, this semigroup need not be DSC. We now introduce some additional conditions which will then imply DSC.

\begin{definition}
    Let $A,B$ and $C$ be non-empty sets, and let $P=(p_{ab})_{a\in A,b\in B}$ be an $A\times B$ matrix with entries from $C$. We say that $P$ is 
   2-\emph{transitive} if the following hold:
    \begin{enumerate}
        \item for every $a_1 \neq a_2 \in A$ and $c_1,c_2 \in C$ there exists $b \in B$ such that $p_{a_1,b} = c_1$ and $p_{a_2,b} = c_2$;
        \item for every $b_1 \neq b_2 \in B$ and $c_1,c_2 \in C$ there exists $a \in A$ such that $p_{a,b_1} = c_1$ and $p_{a,b_2} = c_2$.
    \end{enumerate}
\end{definition}

We will be interested in $2$-transitive $A \times B$ matrices with entries in $W = A \cup B \cup S$. As $|W| \geq |A|,|B|$,  the sets $A$ and $B$ will of necessity be infinite.
For an explicit construction when $A$ and $B$ are countably infinite see \cite{QR2010}.

The proof of the following result closely follows Byleen's proof showing that $\CS$ is congruence-free.
Our proof will be divided into more cases as we have to work around the parts that use symmetry and transitivity.

\begin{thm}\label{byleen}
    The monoid $\CS$, with $\alpha,\beta$ faithful monoid actions and $P$ a 2-transitive matrix over $W = A \cup B \cup S$, has only two diagonal subsemigroups.
\end{thm}

\begin{proof}
    Let $T = \CS$. We will show that $\Delta=\{ (t,t)\colon t\in T\}$ and $T \times T$ are the only diagonal subsemigroups of $T \times T$.
    To this end we will consider arbitrary distinct $vsu, ytx \in T$ and show that the subsemigroup $\rho$ of $T\times T$ generated by $(vsu,ytx)$ and $\Delta$ is equal to $T \times T$.
    Note that if $W \times W \subseteq \rho $, then for any $w = w_1 \dots w_n $ and $ w^\prime = w^\prime_1 \dots w^\prime_m \in T$ we have that $(w_i,1),(1,w_j^\prime) \in \rho$ for all $i,j$, and so
    \begin{equation*}
        (w,w^\prime) = (w_1 , 1)\dots (w_n,1)(1,w_1^\prime)\dots (1,w_m^\prime)  \in \rho
    \end{equation*}
    Hence, under this assumption, $\rho = T \times T$.
    So it suffices to show that $W \times W \subseteq \rho$.
    To do this we will first prove several intermediate claims.

    \begin{claim}
        \label{1}
        For every $u\in A^\ast$  
       there exists $\lambda \in \Delta$ such that $(1,1) = (u,u)\lambda$.
    \end{claim}

    \begin{proof}
        The result is trivial if $u = \epsilon$ so let $u = a_1 \dots a_n$.
        By $2$-transitivity of $P$ there is $b_1 , \dots , b_n \in B$ such that $a_1 b_1 = 1 , a_2 b_2 = b_1 , \dots , a_n b_n = b_{n-1}$. Then if we let $\lambda = (b_n,b_n)$ we have:
        \begin{align*}
            (1,1) = (a_1 \dots a_n b_n ,a_1 \dots a_n b_n) = (u,u) \lambda .
        \end{align*}
    \end{proof}

    \begin{claim}
        \label{2}
        Let $u, x\in A^\ast$ be distinct. 
        Then there exist $\lambda \in \Delta$ and $\epsilon\neq p \in A^*$ such that $(1,p) = (u,x) \lambda$ or $(p,1) = (u,x) \lambda$.
    \end{claim}

    \begin{proof}
        First we note that if either of $u$ or $x$ are empty then the result follows by taking $\lambda = (1,1)$.
        We will use induction on $|u|+|x|$.
        As $u$ and $x$ are distinct, the base case is when $|u|+|x| = 1$, so one of $u,x$ is the empty word. 
        The result then follows by the observation at the start of the proof.
        Let $n > 1$ and suppose for all distinct words $u^\prime,x^\prime \in A^\ast$ with $|u^\prime| + |x^\prime| < n$ that there exists $\lambda \in \Delta$ such that $(u^\prime,x^\prime) \lambda = (1,p)$ or $(p,1)$.
        Now suppose that $|u| +|x| = n$ and that neither $u$ nor $x$ is empty.
        Write $u = a_1 \dots a_n$,  $x = a_1^\prime \dots a_m^\prime$.
        
        If $a_n = a_m^\prime$ then from Claim \ref{1}, there is $\lambda_1 \in \Delta$ such that $(1,1) = (a_n,a_m^\prime) \lambda_1$.
        The words $a_1 \dots a_{n-1} $ and $a_1^\prime \dots a_{m-1}^\prime$ are distinct as $u,x$ are distinct and $a_n = a_m^\prime$.
        By the inductive hypothesis there are $\lambda_2 \in \Delta$ and $\epsilon\neq p \in A^\ast$ such that $(a_1\dots a_{n-1}, a_1^\prime \dots a_{m-1}^\prime) \lambda_2 = (1,p)$ or $(1,p)$. Now take $\lambda = \lambda_1 \lambda_2$.

        Now assume $a_n \neq a_m^\prime$.
        Assume also that $|u| \geq |x|$, the case when $|u| \leq |x|$ is dual.
        There exists $b \in B$ such that $a_n b = a_n,  a_m^\prime b = 1$. Now $|u| > |a_1^\prime \dots a_{m-1}^\prime |$ so the words $u$ and $a_1^\prime \dots a_{m-1}^\prime$ are distinct.
        So by the inductive hypothesis there exist $\lambda^\prime \in \Delta$ and $\epsilon\neq p \in A^\ast$ such that $(u,a_1^\prime \dots a_{m-1}^\prime) \lambda^\prime = (1,p)$ or $(1,p)$.
        Now take $\lambda = (b,b) \lambda^\prime$.
    \end{proof}

    \begin{claim}
        \label{3}
        Let $u, x\in A^\ast$ and $w_1,w_2\in W$, with $u\neq x$.         Then there exist $\lambda, \mu \in \Delta$ such that $(w_1,w_2) = \mu (u,x) \lambda$
    \end{claim}

    \begin{proof}
        By Claim \ref{2}, there exists $\lambda^\prime \in \Delta$ and $\epsilon\neq p\in A^\ast$ such that $(u,x)\lambda^\prime = (1,p)$ or $(p,1)$.
        We will assume $(u,x)\lambda^\prime = (1,p)$, the case when $(u,x)\lambda^\prime = (p,1)$ is dual.
        Let $p = a_1 \dots a_n$.
        There exists $b_0,\dots,b_n \in B$ such that $a_1 b_1 = b_0 , \dots , a_{n-1} b_{n-1} = b_{n-2} , a_n b_n = b_{n-1}$ and $b_n \neq b_0$, so $pb_n = b_0$.
        Now we can pick $a \in A$ such that $ab_n = w_1$ and $ab_0 = w_2$.
        If we let $\mu = (a,a)$ and $\lambda = \lambda^\prime (b_n , b_n)$ then we have:
        \begin{align*}
            \mu (u,x) \lambda  = (a,a)(1,p)(b_n,b_n) = (a b_n , a b_0) = (w_1 , w_2) .
        \end{align*}
    \end{proof}

       The next three claims are dual to Claims \ref{1}, \ref{2}, \ref{3} and we omit their proofs.

    \begin{claim}
        \label{4}
        For every $v\in B^\ast$ there exists $\mu \in \Delta$ such that $(1,1) = \mu (v,v)$.\qed
    \end{claim}

    \begin{claim}
        \label{5}
        Let $v , y\in B^\ast$ be distinct.
        Then there exist $\mu \in \Delta$ and $\epsilon\neq q \in B^\ast$ such that $(1,q) = \mu (v,y)$ or $(q,1) = \mu (v,y)$.\qed
    \end{claim}

    \begin{claim}
        \label{6}
        Let $v,y\in B^\ast$ and $w_1,w_2\in W$, with $v$ and $y$ distinct.
        Then there exist $\lambda,\mu \in \Delta$ such that $(w_1,w_2) = \mu (v,y) \lambda$.\qed
    \end{claim}

    Now let $vsu,ytx \in T$ be distinct, and let $w_1,w_2 \in W$ be arbitrary.
    We will show $(w_1,w_2) \in \rho = \bigl\langle (vsu,ytx), \Delta \bigr\rangle$. 

    If $u = x$ and $v = y$ then it must be the case that $s \neq t$. 
    By Claims \ref{1} and \ref{4} there exist $\lambda, \mu \in \Delta$ such that $(u,x)\lambda = (1,1) = \mu (v,y)$.
    As the right action $\alpha$ is faithful, there is $a \in A$ such that $a_1 = a \triangleright s \neq a \triangleright t = a_2$.
    From Claim \ref{3}, there exist $\lambda^\prime , \mu^\prime \in \Delta$ such that $\mu^\prime (a_1,a_2) \lambda^\prime = (w_1,w_2)$.
    Now we have:
    \begin{equation*}
        (w_1,w_2) = \mu^\prime (a,a) \mu(vsu,ytx)\lambda \lambda^\prime \in \rho.
    \end{equation*}
    Now suppose that $u \neq x$ and $v = y$ (the case when $u = x$ and $v \neq y$ is dual). 
    By Claims \ref{2} and \ref{4} there exist $\lambda , \mu \in \Delta$ and $\epsilon\neq p\in A^\ast$ such that $\mu (v,y) = (1,1)$ and $(u,x) \lambda = (p,1)$ (again the case when $(u,x) \lambda = (1,p)$ is dual).
    Let $a$ be any element of $A$ and let $a_1 = a \triangleright s, a_2 = a \triangleright t$.
    As the words $p$ and $1 = \epsilon$ are distinct, the words $a_1 p$ and $a_2$ are also distinct.
    Hence, by Claim \ref{3}, there exist $\lambda^\prime , \mu^\prime \in \Delta$ such that $\mu^\prime (a_1 p ,a_2) \lambda^\prime = (w_1,w_2)$.
    Now we have:
    \begin{equation*}
        (w_1,w_2) = \mu^\prime (a,a) \mu (vsu,ytx) \lambda \lambda^\prime \in \rho
    \end{equation*}
    If instead we had that $u \neq x$ and $v \neq y$ then, by Claims \ref{2} and \ref{5}, there exist 
    ${\lambda,\mu \in \Delta}$ and $\epsilon\neq p \in A^\ast$, $\epsilon\neq q \in B^\ast$ such that 
    \[
    (u,x) \lambda \in\bigl\{ (1,p),(p,1)\bigr\}\quad\text{and}\quad \mu (v,y) = \bigl\{ (1,q),(q,1)\bigr\}.
    \]

    Here we will write $(u,x) \lambda = (p_1,p_2)$, noting that $| p_1 | \neq | p_2 |$.
    We will treat the case when $\mu(v,y) = (1,q)$, the other case ($(v,y) \mu = (q,1)$) is dual to this.

    Write $q = b_1 \dots b_n$, and let $a$ be any element of $A$.
    There exist $a_n,\dots, a_1 \in A$ such that 
    \[
    a_n b_n = a , a_{n-1} b_{n-1} = a_n , \dots ,a_1 b_1 = a_2.
    \]
    So we have $a_1 q = a$.
    The words $(a_1 \triangleright s) p_1$ and $(a \triangleright t) p_2$ are distinct (as $p_1$ and $p_2$ have different lengths), so, by Claim \ref{3}, there are $\lambda^\prime$, $\mu^\prime \in \Delta$ such that 
    \[
    \mu^\prime ( (a_1 \triangleright s) p_1 , (a \triangleright t) p_2 ) \lambda^\prime = (w_1,w_2).
    \]
     Hence we have
    \begin{equation*}
        (w_1,w_2) = \mu^\prime (a_1,a_1) \mu (vsu ,ytx) \lambda  \lambda^\prime \in \rho.
    \end{equation*}
    
    So in all cases $(w_1,w_2) \in \rho$ and hence $W \times W \subseteq \rho$.
\end{proof}

From here on when we refer to $\CS$, we will assume that $\alpha $ and $\beta$ are faithful actions, and that $P$ is $2$-transitive over $A \cup B \cup S$.

Here we will show that there is non-group DSC semigroups that are bisimple.
The Green's equivalence $\mathcal{D}$ on a semigroup $S$ is defined to be $\mathcal{R} \circ \mathcal{L} = \mathcal{L} \circ \mathcal{R}$, for a more detailed explanation see \cite[Section 2.1]{Ho95}.
We say a semigroup $S$ is \emph{bisimple} if $\mathcal{D} = S \times S$.

\begin{cor}\label{cor:D}
    The following statements involving DSC semigroups hold:
    \begin{thmenumerate}
        \item\label{it:D1} There exist non-group DSC semigroups.
        \item\label{it:D2} There exist non-group DSC semigroups that are regular and bisimple.
        \item\label{it:D3} Subsemigroups of DSC semigroups  are not necessarily DSC.
    \end{thmenumerate}
\end{cor}

\begin{proof}
    \ref{it:D1} From Theorem \ref{byleen}, we have seen that $\CS$ is DSC. 
    To see that it is not a group, take any $a \in A$ and $vsu \in \CS$. Then $vsu a \neq 1$.
    So, the element $a$ has no group theoretic inverse, and hence
   $\CS$ is not a group.

   \ref{it:D2}
    If $S$ is bisimple, Byleen showed in \cite{Byleen} that $\CS$ is also bisimple. 
    Now suppose $S$ is regular. Let $vsu \in \CS$, with $v = b_1 \dots b_m$ and $u = a_1 \dots a_n$.
    There exist $s^\prime \in S$, $a_1^\prime , \dots , a_m^\prime \in A$ and $b_1^\prime , \dots , b_n^\prime \in B$ such that $s s^\prime s = s, s^\prime s s^\prime = s^\prime$ and
    \[ 
    a_1^\prime b_1 =  \dots = a_m^\prime b_m=a_1 b_1^\prime = \dots =a_n b_n^\prime =1.
    \] 
    If we set 
    \[
    y = b_n^\prime \dots b_1^\prime \quad\text{and}\quad  x = a_m^\prime \dots a_1^\prime,
    \]
     then $ys^\prime x$ is an inverse of $vsu$.
    Hence $\CS$ is regular.
    Therefore if a monoid $S$ is regular and bisimple (such as any group) then $\CS$ is regular and bisimple.
    
   \ref{it:D3}
    Consider the subsemigroup $A^\ast$ of $\CS$. This semigroup is not DSC (as it is not simple). Hence $\CS$ has non-DSC subsemigroups.
\end{proof}

\section{Closing Remarks and Further Questions}

Now that we have seen that there exist DSC semigroups that are not groups, one might want to try to understand them better, and perhaps completely classify. This seems to be out of reach at present, and some seemingly easy questions remain. For example, we have even seen that not even all (infinite) groups are DSC semigroups. So one may ask whether a description of DSC groups might be possible. Also, we do not know whether DSC semigroups are closed under formation of direct products.

Another direction one may take is to investigate more systematically the degree of interdependence of all four defining properties of a congruence. Specifically, a relation $\rho$ on a semigroup $S$ is a congruence if and only if it is reflexive, symmetric, transitive and compatible. DSC semigroups are precisely those for which reflexivity and compatibility imply symmetry and transitivity. What about other combinations of these properties?

\bibliography{Reflexive}
\bibliographystyle{abbrv}


\end{document}